\documentclass[11pt]{article}
\usepackage{color}
\usepackage{amsfonts}
\usepackage{amscd}
\usepackage{amssymb}
\usepackage{amsthm}
\usepackage{amsmath, xspace}
\usepackage[all]{xy}
\usepackage{graphics}
\usepackage{graphicx}
\usepackage{enumitem}
\usepackage{fancyhdr}
\usepackage{lscape}

\theoremstyle{plain}
\newtheorem{thm}{Theorem}[section]

\newtheorem{prop}[thm]{Proposition}

\theoremstyle{definition}

\newtheorem{remark}[thm]{Remark}
\theoremstyle{example}

\theoremstyle{remark}

\numberwithin{equation}{section}

\providecommand{\keywords}[1]{\textbf{\textit{Key words---}} #1}

\def\cF{\mathcal{F}}

\def\cL{\mathcal{L}}

\def\cO{\mathcal{O}}
\def\cP{\mathcal{P}}

\def\CC{\mathbb{C}}

\def\KK{\mathbb{K}}

\def\RR{\mathbb{R}}

\def\ZZ{\mathbb{Z}}

\def\fa{\mathfrak{a}}
\def\fb{\mathfrak{b}}

\def\fg{\mathfrak{g}}

\def\fk{\mathfrak{k}}

\def\Card{\mathrm{Card}}

\def\Hom{\mathrm{Hom}}

\def\Ind{\mathrm{Ind}}

\def\wt{\mathrm{wt}}

\usepackage{etex} 
\usepackage{pictexwd}
\usepackage{tikz}
\usepgflibrary{shapes.geometric}
\usetikzlibrary{arrows}

\tikzstyle{V}=[draw, fill =black, circle, inner sep=0pt, minimum size=1.5pt]
\tikzstyle{wV}=[draw, fill =white, circle, inner sep=0pt, minimum size=4.5pt]
\tikzstyle{bV}=[draw, fill =black, circle, inner sep=0pt, minimum size=4.5pt]
\tikzstyle{over}=[draw=white,double=black,line width=2pt, double distance=.5pt]

\def\Over[#1,#2][#3,#4]{ 
	\draw[style=over]   (#1,#2) .. controls ++(0,#4*.5-#2*.5) and ++(0,-#4*.5+#2*.5) .. (#3,#4);}
\def\Under[#1,#2][#3,#4]{ 
	\draw  (#1,#2) .. controls ++(0,#4*.5-#2*.5) and ++(0,-#4*.5+#2*.5) .. (#3,#4);}
\def\Cross[#1,#2][#3,#4]{
	\Under[#3,#2][#1,#4]\Over[#1,#2][#3,#4]}

\def\Tops[#1][#2][#3]{
	\foreach\x in {#1}{
		\draw (\x+.1,#2) -- (\x+.1,#2+.15) (\x-.1,#2) -- (\x-.1,#2+.15) ;
		\draw (\x+.1,#2+.15) arc (0:360:1mm and .5mm);}
	\foreach \x in {1,...,#3} {\draw (\x,#2)  to (\x,#2+.05) node[V]{};}
	}
\def\Bottoms[#1][#2][#3]{
	\foreach\x in {#1}{
		\draw (\x+.1,#2) -- (\x+.1,#2-.1) (\x-.1,#2) -- (\x-.1,#2-.1) ;
		\draw (\x+.1,#2-.1) arc (0:-180:1mm and .5mm);}
	\foreach \x in {1,...,#3} {\draw (\x,#2)  to (\x,#2-.05) node[V]{};}
	}
\def\Caps[#1][#2,#3][#4]{
	\Tops[#1][#3][#4]
	\Bottoms[#1][#2][#4]
	}
\def\Pole[#1][#2,#3]{
	\shade[left color=white,right color=white] (#1+.1,#2) rectangle (#1-.1,#3);
	\draw[over] (#1+.1,#2) to (#1+.1,#3) (#1-.1,#2) to (#1-.1,#3) ;}
\def\Label[#1,#2][#3][#4]{
	\node[above] at (#3,#2+.1) {#4};
	\node[below] at (#3,#1-.1) {#4};		}

\makeatletter
\renewcommand{\@makefnmark}{\mbox{\textsuperscript{}}}
\makeatother

\title{The Steinberg-Lusztig tensor product theorem, Casselman-Shalika and LLT polynomials}
\author{
Martina Lanini \ \ email:\ lanini@mat.uniroma2.it \\
Arun Ram\quad\ \ email:\ aram@unimelb.edu.au \\
\\
}
\date{}

\begin{document}

\maketitle

\begin{center}
{\sl Dedicated to Friedrich Knop and Peter Littelmann \\
 on the occasion of their 60th birthdays}
\end{center}

\begin{abstract}
\noindent
In this paper we establish a Steinberg-Lusztig tensor product theorem for
abstract Fock space.  This is a generalization of the type A result of Leclerc-Thibon
and a Grothendieck group version of the Steinberg-Lusztig tensor product theorem
for representations of quantum groups at roots of unity.  Although the statement can 
be phrased in terms of parabolic affine Kazhdan-Lusztig polynomials and thus has 
geometric content, our proof is combinatorial,
using the theory of crystals (Littelmann paths).  We derive the Casselman-Shalika formula
as a consequence of the Steinberg-Lusztig tensor product theorem for
abstract Fock space.
\end{abstract}


\setcounter{section}{-1}

\section{Introduction}

In our previous paper \cite{LRS} we provided a construction of an ``abstract'' Fock space $\cF_\ell$ in a general Lie type setting.
The construction is given by simple combinatorial ``straightening relations'' which
generalize the Kashiwara-Miwa-Stern \cite{KMS} formulation of the $q$-Fock space from the
type A case. We showed that the abstract Fock space is a combinatorial realization of the 
graded Grothendieck group of finite dimensional representations of the quantum group at
a root of unity, where the standard basis elements $\vert\lambda \rangle$
correspond to the Weyl modules $\Delta_q(\lambda)$ and the KL-basis $C_\lambda$
corresponds to the simple modules $L_q(\lambda)$.

In Section 1 we prove a product theorem (Theorem \ref{TenPrdThm})
in abstract Fock space which generalizes the type A
theorem of Leclerc and Thibon \cite[Theorem 6.9]{LT}.  Our proof 
follows the same pattern as the proof for Type A given in \cite[Theorem 6.9]{LT}
except that, in order to deal with general Lie type,
we have replaced the use of ribbon tableaux 
with the crystal basis and Littelmann paths.
The basic philosophy of our technique is similar to the main idea of a paper of Guilhot \cite{Gu}
but we also make use of the elegant cancellation technique of 
Littelmann \cite[proof of Theorem 9.1]{Li} to complete the proof.  This technique provides
a combinatorial control of the Demazure operator used in the proof of \cite[Lemma 4.4]{Knp}.
We have not considered the unequal parameter case in this paper but the close relation
between our context and that of \cite{Knp} cries out for an extension of the tensor 
product theorem for abstract Fock space to unequal parameters.


The Casselman-Shalika formula is important in the representation theory of
p-adic groups (see \cite{CS}), in its relation
to the affine Hecke algebra (see for example \cite{BBF}) and in the 
geometric Langlands program (see \cite{FGV} and \cite{NP}).
In Section 2 we show that the Casselman-Shalika formula can be
derived as a special case of the Steinberg tensor product theorem for abstract Fock space.
This derivation is done by using the relationship between the abstract Fock space and
the affine Hecke algebra as detailed in \cite[Theorem 4.7]{LRS}.

In Section 3, we review the connection between the abstract Fock space and the 
representations of quantum groups at roots of unity (Theorem 3.1) and the Steinberg-Lusztig
tensor product theorem (Theorem 3.2).  The Steinberg-Lusztig tensor product theorem is the 
primary motivation for the product theorem in abstract Fock space.  Our approach does provide 
an alternative proof of the  Steinberg-Lusztig tensor product theorem for representations of 
quantum groups at  roots of unity (though hardly elementary since proving the 
Steinberg-Lusztig tensor product theorem this way relies on deep results of Kazhdan-Lusztig
\cite{KL94} 
and Kashiwara-Tanisaki \cite{KT95}).

As explained in \cite{LT}, the Steinberg-Lustzig tensor product theorem and 
the abstract Fock space are intimately related to the
LLT polynomials defined in type A by Lascoux, Leclerc and Thibon \cite{LLT}.  
Fundamentally, the LLT polynomials are taking the role of the characters of the
Frobenius twisted Weyl modules which, by the Steinberg-Lusztig tensor product theorem, are
simple modules for the quantum group at a root of unity.
General Lie type definitions
of LLT polynomials have been given by Grojnowski-Haiman \cite{GH} 
and Lecouvey \cite{Lcy}.
In the second half of Section 3, we summarize a 2008 letter 
from C.\ Lecouvey to A.\ Ram which explains 
that a consequence of the tensor product theorem for 
abstract Fock space is that the definition from \cite{GH} and the definition from \cite{Lcy} 
coincide up to a  power of $t^{\frac12}$.  

Kazhdan and Lusztig \cite{KL94} established an equivalence of categories between an appropriate 
category of representations of the affine Lie algebra (of negative level) and the finite dimensional
representations of the quantum group (of the finite dimensional Lie algebra) at a root of unity.
In Section 4 we review this correspondence and make explicit the tensor product theorem
in terms of representations of the affine Lie algebra.  This produces a character formula for certain
negative level irreducible highest weight representations of the affine Lie algebra.  From the point of
view of this paper this character formula is an easy consequence of \cite{KL94} and
\cite{Lu89}.  We find it difficult to believe that this formula has not been noticed before but we
have not yet been able to locate a suitable specific reference.

We thank all the institutions which have supported our work on this paper, 
particularly the Univ. of Melbourne,  the 
Australian Research Council (grants DP1201001942 and DP130100674) 
and  ICERM (Institute for Computational and Experimental Research in Mathematics).  
We thank Kari Vilonen and Ting Xue for generous support of a visit of Martina Lanini to 
University of Melbourne funded by their Australian Research Council grant DP150103525.

\section{A product theorem in abstract Fock space $\cF_\ell$}\label{sectionFelldefin}

Let $W_0$ be a finite Weyl group, generated by simple reflections
$s_1, \ldots, s_n$, and acting on a lattice of weights $\fa_\ZZ^*$.  
For example, this situation arises when $T$ is a maximal torus of a reductive algebraic group $G$,
\begin{equation}
\fa_\ZZ^* = \Hom(T,\CC^\times)
\qquad\hbox{and}\qquad
W_0 = N(T)/T,
\label{wtsWeylgpdefn}
\end{equation}
where $N(T)$ is the normalizer of $T$ in $G$.  The simple reflections
in $W_0$ correspond to a choice of Borel subgroup $B$ of $G$ which contains $T$.
Let $R^+$ denote the positive roots.  Let $\alpha_1, \ldots, \alpha_n$ be the simple roots and let
$\alpha_1^\vee, \ldots, \alpha_n^\vee$ be the simple coroots.
The \emph{dot action} of $W_0$ on $\fa_\ZZ^*$ is given by
\begin{equation}
w\circ\lambda = w(\lambda+\rho)-\rho,
\qquad\hbox{where}\quad \rho = \hbox{$\frac12$}\sum_{\alpha\in R^+} \alpha
\label{dotaction}
\end{equation}
is the half sum of the positive roots for $G$ (with respect to $B$).

Fix $\ell\in \ZZ_{>0}$.  The \emph{abstract Fock space} $\mathcal{F}_\ell$ is
the $\ZZ[t^{\frac12},t^{-\frac12}]$-module generated by $\{ \vert \lambda\rangle \ |\ \lambda\in \fa_{\ZZ}^{*}\}$ with relations
\begin{equation}
\vert s_i\circ\lambda\rangle=\begin{cases}
-\vert \lambda\rangle, &\hbox{if $\langle\lambda+\rho,\alpha_i^\vee\rangle \in \ell\ZZ_{\ge 0}$,} \\
-t^{\frac12}\vert \lambda\rangle, &\hbox{if $0<\langle\lambda+\rho,\alpha_i^\vee\rangle<\ell $,} \\
-t^{\frac12}\vert s_i\circ\lambda^{(1)} \rangle -\vert \lambda^{(1)} \rangle -t^{\frac12}\vert \lambda\rangle, 
&\hbox{if $ \langle\lambda+\rho,\alpha_i^\vee\rangle > \ell$ and $\langle\lambda+\rho,\alpha_i^\vee\rangle\not\in \ell\ZZ$,}
\end{cases}
\label{Fstraightening}
\end{equation}
where 
$\lambda^{(1)} = \lambda - j\alpha_i$
if $\langle\lambda+\rho, \alpha_i^\vee\rangle = k\ell + j$ with $k\in \ZZ_{>0}$ and $j\in \{1, \ldots, \ell-1\}$.

The following picture illustrates the terms in \eqref{Fstraightening}.  This is the case $G=SL_2$ with $\ell=5$, 
$\langle\omega_1,\alpha_1^\vee\rangle=1$ and $\alpha_1=2\omega_1$ and, in the picture,
$\lambda$ corresponds to the third case of \eqref{Fstraightening}, $\mu$ to the first case and $\nu$ to the second case.
$$\setlength{\unitlength}{0.5cm}
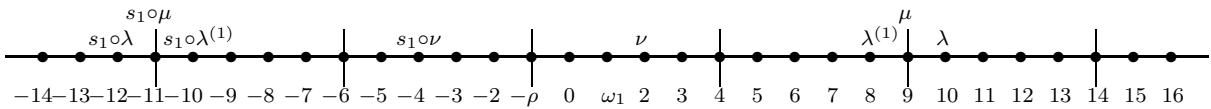
\begin{picture}(23,3)
\put(-5,2){\line(1,0){32}}

\put(-4,2){\circle*{0.3}}
\put(-3,2){\circle*{0.3}}
\put(-2,2){\circle*{0.3}}
\put(-1,2){\circle*{0.3}}
\put(0,2){\circle*{0.3}}
\put(1,2){\circle*{0.3}}
\put(2,2){\circle*{0.3}}
\put(3,2){\circle*{0.3}}
\put(4,2){\circle*{0.3}}
\put(5,2){\circle*{0.3}}
\put(6,2){\circle*{0.3}}
\put(7,2){\circle*{0.3}}
\put(8,2){\circle*{0.3}}
\put(9,2){\circle*{0.3}}
\put(10,2){\circle*{0.3}}
\put(11,2){\circle*{0.3}}
\put(12,2){\circle*{0.3}}
\put(13,2){\circle*{0.3}}
\put(14,2){\circle*{0.3}}
\put(15,2){\circle*{0.3}}
\put(16,2){\circle*{0.3}}
\put(17,2){\circle*{0.3}}
\put(18,2){\circle*{0.3}}
\put(19,2){\circle*{0.3}}
\put(20,2){\circle*{0.3}}
\put(21,2){\circle*{0.3}}
\put(22,2){\circle*{0.3}}
\put(23,2){\circle*{0.3}}
\put(24,2){\circle*{0.3}}
\put(25,2){\circle*{0.3}}
\put(26,2){\circle*{0.3}}

\put(-4.8,0.8){$\scriptstyle{-14}$}
\put(-3.8,0.8){$\scriptstyle{-13}$}
\put(-2.8,0.8){$\scriptstyle{-12}$}
\put(-1.8,0.8){$\scriptstyle{-11}$}
\put(-0.8,0.8){$\scriptstyle{-10}$}
\put(0.4,0.8){$\scriptstyle{-9}$}
\put(1.4,0.8){$\scriptstyle{-8}$}
\put(2.4,0.8){$\scriptstyle{-7}$}
\put(3.4,0.8){$\scriptstyle{-6}$}
\put(4.4,0.8){$\scriptstyle{-5}$}
\put(5.4,0.8){$\scriptstyle{-4}$}
\put(6.4,0.8){$\scriptstyle{-3}$}
\put(7.4,0.8){$\scriptstyle{-2}$}
\put(8.4,0.8){$\scriptstyle{-\rho}$}
\put(9.85,0.8){$\scriptstyle{0}$}
\put(10.85,0.8){$\scriptstyle{\omega_1}$}
\put(11.85,0.8){$\scriptstyle{2}$}
\put(12.85,0.8){$\scriptstyle{3}$}
\put(13.85,0.8){$\scriptstyle{4}$}
\put(14.85,0.8){$\scriptstyle{5}$}
\put(15.85,0.8){$\scriptstyle{6}$}
\put(16.85,0.8){$\scriptstyle{7}$}
\put(17.85,0.8){$\scriptstyle{8}$}
\put(18.85,0.8){$\scriptstyle{9}$}
\put(19.75,0.8){$\scriptstyle{10}$}
\put(20.75,0.8){$\scriptstyle{11}$}
\put(21.75,0.8){$\scriptstyle{12}$}
\put(22.75,0.8){$\scriptstyle{13}$}
\put(23.75,0.8){$\scriptstyle{14}$}
\put(24.75,0.8){$\scriptstyle{15}$}
\put(25.75,0.8){$\scriptstyle{16}$}

\put(-1,1.2){\line(0,1){1.5}}
\put(4,1.2){\line(0,1){1.5}}
\put(9,1.2){\line(0,1){1.5}}
\put(14,1.2){\line(0,1){1.5}}
\put(19,1.2){\line(0,1){1.5}}
\put(24,1.2){\line(0,1){1.5}}


\put(19.75,2.3){$\scriptstyle{\lambda}$}
\put(17.75,2.3){$\scriptstyle{\lambda^{(1)}}$}
\put(-0.8,2.3){$\scriptstyle{s_1\circ\lambda^{(1)} }$}
\put(-2.8,2.3){$\scriptstyle{s_1\circ\lambda}$}

\put(18.75,3){$\scriptstyle{\mu}$}
\put(-1.8,3){$\scriptstyle{s_1\circ\mu}$}

\put(11.75,2.3){$\scriptstyle{\nu}$}
\put(5.4,2.3){$\scriptstyle{s_1\circ\nu}$}


\end{picture}
$$
Define a $\ZZ$-linear involution 
$\overline{\phantom{T}}\colon \mathcal{F}_{\ell}\to \mathcal{F}_{\ell}$ by
\begin{equation}
\overline{t^{\frac12}} = t^{-\frac12}
\qquad\hbox{and}\qquad
\overline{\vert \lambda \rangle} = (-1)^{\ell(w_0)}(t^{-\frac12})^{\ell(w_{0})-N_\lambda}\, \vert w_0\circ \lambda \rangle.
\label{Fellbar}
\end{equation}
where $w_0$ is the longest element of $W_0$, $\ell(w_0) = \Card(R^+)$ is the length of $w_0$, and 
$N_\lambda = \Card\{\alpha\in R^+\ |\ \langle \lambda+\rho, \alpha^\vee\rangle \in \ell\ZZ\}.$

The \emph{dominant integral weights}  with the \emph{dominance partial order} $\le$ are the elements of 
\begin{equation}
\begin{array}{c}
(\fa_\ZZ^*)^+ 
= \{\lambda\in \fa_\ZZ^*\ |\ \hbox{$\langle \lambda+\rho, \alpha_i^\vee\rangle> 0$ for $i=1, 2,\ldots, n$} \}
\\
\\
\hbox{with}\qquad
\mu\le \lambda\qquad\hbox{if}\quad
\mu \in \lambda - \sum_{\alpha\in R^+} \ZZ_{\ge 0}\alpha.
\end{array}
\label{domintwtsdefn}
\end{equation}
In \cite[Theorem 1.1 and Proposition 2.1]{LRS} we showed that
$\cF_\ell$ has bases
\begin{equation}
\{ \vert\lambda\rangle\ |\ \lambda\in (\fa_\ZZ^*)^+\}
\qquad\hbox{and}\qquad
\{ C_\lambda\ |\ \lambda\in (\fa_\ZZ^*)^+\}
\label{Fellbases}
\end{equation}
where $C_\lambda$ are determined by 
\begin{equation}
\overline{C_\lambda} = C_\lambda
\qquad\hbox{and}\qquad
C_\lambda = \vert \lambda\rangle + \sum_{\mu\ne \lambda} p_{\mu\lambda} \vert\mu\rangle,
\qquad\hbox{with $p_{\mu\lambda}\in t^{\frac12}\ZZ[t^{\frac12}]$.}
\label{Clambdadefn}
\end{equation}

\subsection{The action of $\KK[X]^{W_0}$ on $\cF_\ell$}

Letting $\KK = \ZZ[t^{\frac12}, t^{-\frac12}]$, the group algebra of $\fa_\ZZ^*$ is
\begin{equation}
\KK[X] = \hbox{$\KK$-span}\{ X^\mu\ |\ \mu\in \fa_\ZZ^*\}
\quad\hbox{with}\quad
X^\mu X^\nu = X^{\mu+\nu}.
\label{KXdefn}
\end{equation}
The Weyl group $W_0$ acts $\KK$-linearly on $\KK[X]$ by
\begin{equation}
wX^\mu = X^{w\mu},\ \hbox{for $w\in W_0$ and $\mu\in \fa_\ZZ^*$,}
\quad\hbox{and}\quad 
\KK[X]^{W_0} = \{ f\in \KK[X]\ |\ wf = f\}
\end{equation}
is the ring of \emph{symmetric functions}.

Let $V$ be the free $\KK$-module generated 
by $\{ \vert \lambda\rangle \ |\ \lambda\in \fa_{\ZZ}^{*}\}$ so that
\begin{equation}
\cF_\ell \cong V/I,
\label{Fasqtt}
\end{equation}
where $I$ is the subspace of $V$  
consisting of $\KK$-linear combinations of 
the elements
$$
\begin{array}{cl}
a_\lambda = \vert s_i\circ\lambda\rangle+\vert \lambda\rangle, 
&\hbox{with $\langle\lambda+\rho,\alpha_i^\vee\rangle \in \ell\ZZ_{\ge 0}$,} \\
b_\lambda = \vert s_i\circ\lambda\rangle+t^{\frac12}\vert \lambda\rangle, 
&\hbox{with $0<\langle\lambda+\rho,\alpha_i^\vee\rangle<\ell $, and} \\
c_\lambda = \vert s_i\circ\lambda\rangle+t^{\frac12}\vert s_i\circ\lambda^{(1)} \rangle 
+\vert \lambda^{(1)} \rangle + t^{\frac12}\vert \lambda\rangle,\quad 
&\hbox{with $ \langle\lambda+\rho,\alpha_i^\vee\rangle > \ell$ and $\langle\lambda+\rho,\alpha_i^\vee\rangle\not\in \ell\ZZ$.}
\end{array}
$$

Let $\mathring{\fg}$ be the Lie algebra of the reductive group $G$ alluded to in \eqref{wtsWeylgpdefn}.
Let $\varphi$ be the highest weight of the adjoint representation and let
$\varphi^\vee\in [\mathring{\fg}_\varphi, \mathring{\fg}_{-\varphi}]$ such that
$\langle \varphi, \varphi^\vee\rangle=2$ (so that $\varphi^\vee$ is an appropriate normalized
highest short coroot of $\mathring{\fg}$).  
\begin{equation}
\hbox{The \emph{dual Coxeter number} is} \qquad
h = \langle \rho, \varphi^\vee\rangle + 1.
\label{Coxeternumber}
\end{equation}
The \emph{level $(-\ell-h)$ action} of 
$\KK[X]$ on $V$ is the $\KK$-linear extension of
\begin{equation}
X^\mu \cdot\vert \gamma\rangle = \vert -\ell w_0\mu+\gamma\rangle,
\qquad\hbox{for $\mu, \gamma\in \fa_\ZZ^*$.}
\label{Vaction}
\end{equation}
Letting $w_0$ be the longest element of $W_0$, define
$$\mu^* = -w_0\mu
\qquad\hbox{and}\qquad w^*= w_0ww_0,
\qquad\hbox{for $\mu\in \fa_\ZZ^*$ and $w\in W$.}
$$
(This notation is such that if $\mu\in \fa_\ZZ^*$ and $L_{\mathring{\fg}}(\mu)$ denotes the 
irreducible $\mathring{\fg}$-module of 
highest weight $\mu$ then the dual $L_{\mathring{\fg}}(\mu)^*\cong L_{\mathring{\fg}}(\mu^*)$ 
and ${}^*\colon W_0\to W_0$ is the involutive automorphism of $W_0$ 
induced by the automorphism of the Dynkin diagram specified by $s_i^* = s_{i^*}$.)  
Then
\begin{equation}
X^\mu \cdot\vert \gamma\rangle = \vert \ell \mu^*+\gamma\rangle
\qquad\hbox{and}\qquad
(w\mu)^* 
= w^*\mu^*.
\label{staraction}
\end{equation}

The following proposition establishes an action of the ring of symmetric functions
$\KK[X]^{W_0}$ on the abstract Fock space $\cF_\ell$.  From the point of view of
Theorem \ref{abstracttoHecke} below, this action is coming from an action of the center
of the affine Hecke algebra which, by an important result of Bernstein, is the ring
of symmetric functions (inside the affine Hecke algebra).  Our proof of 
Proposition \ref{KXaction} provides an independent proof of the existence of the action
of $\KK[X]^{W_0}$ without referring to the affine Hecke algebra and the characterization of its center.

\begin{prop} \label{KXaction}
The action of $\KK[X]$ on $V$ given in \eqref{Vaction} induces a $\KK$-linear action of
the ring $\KK[X]^{W_0}$ of symmetric functions on $\cF_\ell$  by
$$\left(\sum_{w\in W_0} X^{w\mu}\right) \cdot \vert \gamma\rangle 
= \sum_{w\in W_0} \vert \ell(w\mu)^*+\gamma\rangle,
\quad\hbox{for $\mu\in \fa_\ZZ^*$ and $\gamma\in \fa_\ZZ^*$.}
$$
\end{prop}
\begin{proof}
Let $f$ be an element of the subspace $I$ defined in \eqref{Fasqtt}, let $\mu\in \fa_\ZZ^*$
and let $i\in \{1, \ldots, n\}$.  Summing over a set of representatives of the cosets in $\{1, s_i^*\}\backslash W_0$,
$$\left(\sum_{w\in W_0} X^{w\mu}\right) \cdot f
= \left( \sum_{v\in \{1,s_i^*\}\backslash W_0} (X^{v\mu}+X^{s_i^*v\mu})\right)\cdot f,$$
where the representatives $v\in \{1, s_i^*\}\backslash W_0$ are chosen such that
$\langle v\mu, \alpha_{i^*}^\vee\rangle \in \ZZ_{\ge 0}$.

\smallskip\noindent
Case 1:  $f = \vert s_i\circ\lambda\rangle 
+ \vert \lambda\rangle$ with $\langle \lambda+\rho, \alpha_i^\vee\rangle \in \ell\ZZ_{\geq 0}$.  Then
\begin{align*}
(X^{s_i^*v\mu}&+X^{v\mu})\cdot (\vert s_i\circ\lambda\rangle 
+ \vert \lambda\rangle)\\
&= 
\vert \ell (s_i^*v\mu)^* + s_i\circ\lambda\rangle 
+\vert \ell (v\mu)^* + s_i\circ\lambda\rangle 
+ \vert \ell (s_i^*v\mu)^* + \lambda\rangle
+ \vert \ell (v\mu)^* + \lambda\rangle \\
&= 
\vert \ell s_iv^*\mu^* + s_i\circ\lambda\rangle 
+\vert \ell v^*\mu^* + s_i\circ\lambda\rangle 
+ \vert \ell s_iv^*\mu^* + \lambda\rangle
+ \vert \ell v^*\mu^* + \lambda\rangle \\
&=
\vert s_i\circ( \ell v^*\mu^* + \lambda)\rangle 
+\vert s_i\circ( \ell s_iv^*\mu^* + \lambda)\rangle 
+ \vert \ell s_iv^*\mu^* + \lambda\rangle
+ \vert \ell v^*\mu^* + \lambda\rangle \\
&=  \begin{cases}
a_{\ell v^*\mu^*+s_i\circ \lambda}+a_{\ell v^*\mu^*+\lambda}, 
&\hbox{if $\langle \ell v^*\mu^*, \alpha_i^\vee\rangle > \langle \lambda+\rho, \alpha_i^\vee\rangle$,} \\
a_{\ell s_iv^*\mu^* + \lambda}+a_{\ell v^*\mu^*+\lambda}, 
&\hbox{if $\langle \ell v^*\mu^*, \alpha_i^\vee\rangle \le \langle\lambda+\rho, \alpha_i^\vee\rangle$. }
\end{cases}
\end{align*}
Thus the right hand side is an element of $I$.

\smallskip\noindent
Case 2: $f = \vert s_i\circ\lambda\rangle 
+ t^{\frac12}\vert \lambda\rangle$ with $0<\langle \lambda+\rho, \alpha_i^\vee\rangle < \ell$.
Then
$\ell v^*\mu^*+s_i\circ\lambda = (\ell v^*\mu^*+\lambda)^{(1)}$ so that
\begin{align*}
(X^{s_i^*v\mu}&+X^{v\mu})\cdot (\vert s_i\circ\lambda\rangle 
+ t^{\frac12}\vert \lambda\rangle)\\
&= 
\vert \ell (s_i^* v\mu)^* + s_i\circ\lambda\rangle 
+\vert \ell (v\mu)^* + s_i\circ\lambda\rangle 
+ t^{\frac12}\vert \ell (s_i^* v\mu)^* + \lambda\rangle
+ t^{\frac12}\vert \ell (v\mu)^* + \lambda\rangle \\
&= 
\vert \ell s_i v^*\mu^* + s_i\circ\lambda\rangle 
+\vert \ell v^*\mu^* + s_i\circ\lambda\rangle 
+ t^{\frac12}\vert \ell s_i v^*\mu^* + \lambda\rangle
+ t^{\frac12}\vert \ell v^*\mu^* + \lambda\rangle \\
&= 
\vert s_i\circ ( \ell v^*\mu^* + \lambda) \rangle 
+\vert ( \ell v^*\mu^* + \lambda)^{(1)}\rangle 
+ t^{\frac12}\vert s_i\circ( \ell v^*\mu^* + s_i\circ \lambda)\rangle
+ t^{\frac12}\vert \ell v^*\mu^* + \lambda\rangle \\
&= 
\vert s_i\circ ( \ell v^*\mu^* + \lambda) \rangle 
+ t^{\frac12}\vert s_i\circ( \ell v^*\mu^* + \lambda)^{(1)} \rangle
+\vert ( \ell v^*\mu^* + \lambda)^{(1)}\rangle 
+ t^{\frac12}\vert \ell v^*\mu^* + \lambda\rangle \\
&= \begin{cases}
c_{\ell v^*\mu^*+\lambda}, &\hbox{if $\langle v^*\mu^*, \alpha_i^\vee\rangle\in \ZZ_{>0}$, } \\
2b_{\ell v^*\mu^* + \lambda}, &\hbox{if $\langle v^*\mu^*, \alpha_i^\vee\rangle =0$,}
\end{cases}
\end{align*}
since if $s_i^*v\mu\ne v\mu$ then $s_iv^*\mu^*\ne v^*\mu^*$ and 
$\langle v^*\mu^*, \alpha_i^\vee\rangle\in \ZZ_{>0}$ then
$\langle \ell v^*\mu^*+\lambda+\rho, \alpha_i^\vee\rangle > \ell$ and
$\langle \ell v^*\mu^*+\lambda+\rho, \alpha_i^\vee\rangle \not\in \ell\ZZ$.
Thus the right hand side is an element of $I$.

\smallskip\noindent
Case 3:  
Assume $\lambda\in \fa_\ZZ^*$ with $ \langle\lambda+\rho,\alpha_i^\vee\rangle > \ell$ and $\langle\lambda+\rho,\alpha_i^\vee\rangle\not\in \ell\ZZ$.
If $\mu\in \fa_\ZZ^*$ and $\langle \nu, \alpha_i^*\rangle\in \ZZ_{\ge 0}$ then
$$s_i\circ (v^*\mu^*+\nu) = s_i(v^*\mu^*+\nu+\rho)-\rho = s_iv^*\mu^* + s_i\circ\nu
\quad\hbox{and}\quad
(\ell \nu^* + \lambda)^{(1)} = \ell \nu^* + \lambda^{(1)},$$
so that, with $\langle \lambda+\rho, \alpha_i^\vee\rangle = k\ell+j$ with $k>0$ and $0\le j<\ell$,
\begin{align*}
(&X^{s_i^*v\mu}+X^{v\mu})\cdot (\vert s_i\circ\lambda\rangle 
+t^{\frac12}\vert s_i\circ \lambda^{(1)}\rangle +\vert \lambda^{(1)}\rangle 
+ t^{\frac12}\vert \lambda\rangle)\\
&= 
\vert \ell (s_i^* v\mu)^* + s_i\circ\lambda\rangle 
+\vert \ell (v\mu)^* + s_i\circ\lambda\rangle 
+t^{\frac12}\vert \ell (s_i^*v\mu)^*+ s_i\circ \lambda^{(1)} \rangle
+t^{\frac12}\vert \ell (v\mu)^*+ s_i\circ \lambda^{(1)} \rangle \\
&\qquad +\vert \ell (s_i^*v\mu)^*+ \lambda^{(1)} \rangle
+\vert \ell (v\mu)^* + \lambda^{(1)} \rangle
+ t^{\frac12}\vert \ell (s_i^* v\mu)^* + \lambda\rangle
+ t^{\frac12}\vert \ell(v\mu)^* + \lambda\rangle \\
&= 
\vert \ell s_iv^*\mu^* + s_i\circ\lambda\rangle 
+\vert \ell v^*\mu^* + s_i\circ\lambda\rangle 
+t^{\frac12}\vert \ell s_i v^*\mu^*+ s_i\circ \lambda^{(1)} \rangle
+t^{\frac12}\vert \ell v^*\mu^*+ s_i\circ \lambda^{(1)} \rangle \\
&\qquad +\vert \ell s_i v^*\mu^*+ \lambda^{(1)} \rangle
+\vert \ell v^*\mu^* + \lambda^{(1)} \rangle
+ t^{\frac12}\vert \ell s_i v^*\mu^* + \lambda\rangle
+ t^{\frac12}\vert \ell v^*\mu^* + \lambda\rangle \\
&= 
\vert s_i\circ(\ell v^*\mu^* + \lambda) \rangle 
+\vert s_i\circ( \ell s_i v^*\mu^* + \lambda)\rangle 
+t^{\frac12}\vert s_i\circ (\ell v^*\mu^* +\lambda^{(1)}) \rangle
+t^{\frac12}\vert s_i\circ(\ell s_i v^*\mu^* + \lambda^{(1)}) \rangle \\
&\qquad +\vert \ell s_i v^*\mu^* + \lambda^{(1)} \rangle
+\vert (\ell v^*\mu^* + \lambda)^{(1)} \rangle
+ t^{\frac12}\vert \ell s_i v^*\mu^* + \lambda\rangle
+ t^{\frac12}\vert \ell v^*\mu^* + \lambda\rangle \\
&=\begin{cases}
c_{\lambda+ \ell v^* \mu^*}+c_{s_i\circ \lambda^{(1)}+v^*\mu^*},
&\hbox{if $\langle \ell v^*\mu^*, \alpha_i^\vee\rangle > \ell k>0$,}
\\
c_{\lambda+ \ell v^* \mu^*}+c_{\lambda+ s_i v^*\mu^*},
&\hbox{if $0<\langle \ell v^*\mu^*, \alpha_i^\vee\rangle < \ell k$,}
\\
c_{\lambda+ \ell v^* \mu^*}
+b_{\lambda+ s_i v^*\mu^*}
+b_{s_i\circ \lambda^{(1)}+v^*\mu^*},
&\hbox{if $\langle \ell v^*\mu^*, \alpha_i^\vee\rangle = \ell k$.}
\end{cases}
\end{align*}
Thus the right hand side is an element of $I$.

These computations show that $I$ is stable under the action of $\KK[X]^{W_0}$.
Thus the action of $\KK[X]^{W_0}$ on $\cF_\ell = V/I$ is well defined.
\end{proof}

 \begin{remark}  One might be tempted to try to define an action of
 $\KK[X]$ on $\cF_\ell$ by $X^\mu\cdot \vert \gamma \rangle = \vert \gamma+\ell\mu\rangle$
 for $\mu, \gamma\in \fa_\ZZ^*$ but this action is not well defined.  For example in the
 $G=SL_2$ case with $\ell=5$ pictured after \eqref{Fstraightening}, one would have
 $0 = X^{\omega_1}\cdot \vert -1\rangle = \vert 5-1\rangle = \vert 4\rangle$, 
 which is a contradiction to \eqref{Fellbases}.  
 On the other hand $0=(X^{-\omega_1}+X^{\omega_1})\cdot \vert -1\rangle
 = \vert -5-1\rangle + \vert 4\rangle = 0$, as it should be.
 \end{remark}

\subsection{The product theorem}

 Let $\mathring{\fg}$ be the Lie algebra of the reductive group $G$ alluded to in \eqref{wtsWeylgpdefn}.
For $\lambda\in (\fa_\ZZ^*)^+$ let $L_{\mathring{\fg}}(\lambda)$ be the irreducible
$U\mathring{\fg}$-module of highest weight $\lambda$ and let $B(\lambda)$ be the 
crystal of $L_{\mathring{\fg}}(\lambda)$,
$$B(\lambda) = \{ \hbox{LS paths $p$ of type $\lambda$}\}
\quad\hbox{and}\quad
\hbox{$\mathrm{wt}(p)$ denotes the endpoint of $p$,}
$$
see \cite[\S5]{Ra}. The \emph{Weyl character} corresponding to $\lambda$ is 
the element of $\KK[X]^{W_0}$ given by
\begin{equation}
s_{\lambda} = \mathrm{char}(L_{\mathring{\fg}}(\lambda))
=\frac{\displaystyle{\sum_{w\in W_0} \det(w) X^{w\circ \lambda} } }
{\displaystyle{ \sum_{w\in W_0} \det(w) X^{w\circ 0} } }
= \sum_{p \in B(\lambda)} X^{\mathrm{wt}(p)}.
\label{Weylchdefn}
\end{equation}

An \emph{$\ell$-restricted} dominant integral weight is $\lambda_0\in (\fa_\ZZ^*)^+$
such that $\langle \lambda_0,\alpha_i^\vee\rangle < \ell$ for $i\in \{1, \ldots, n\}$.
In other words, if $\omega_1, \ldots, \omega_n$ are the fundamental weights for 
$\mathring{\fg}$ then
a weight $\lambda_0\in (\fa_\ZZ^*)^+$ is $\ell$-restricted if $\lambda_0$ is an 
element of 
\begin{equation}
\Pi_\ell = \{ a_1\omega_1+\cdots+a_n\omega_n\ |\ a_1, \ldots, a_n\in \{0,1,\ldots, \ell-1\} \}.
\label{Pielldefn}
\end{equation}

\begin{thm} \label{TenPrdThm}
Let $\lambda\in (\fa_\ZZ^*)^+$ be a dominant integral weight and write
$$\lambda = \ell \lambda_1 + \lambda_0,
\qquad\hbox{with $\lambda_0\in \Pi_\ell$ and $\lambda_1\in (\fa_\ZZ^*)^+$.}
$$
Then, with $C_\lambda\in \cF_\ell$ as in \eqref{Clambdadefn} 
and the $\KK[X]^{W_0}$-action on $\cF_\ell$ as in Proposition \ref{KXaction},
$$C_\lambda = s_{\lambda_1^*} \cdot C_{\lambda_0}.$$
\end{thm}
\begin{proof}
The proof is accomplished in two steps:
\begin{enumerate}
\item[(a)] Show that $s_{\lambda_1^*} \cdot C_{\lambda_0}$ is bar invariant.
\item[(b)] Show that $s_{\lambda_1^*} \cdot C_{\lambda_0 }
= \vert \lambda\rangle + \sum_{\mu\ne \lambda} c_{\mu} \vert\mu\rangle$
with $c_{\mu}\in t^{\frac12}\ZZ[t^{\frac12}]$.
\end{enumerate}
Proof of (a):  The bar involution and $N_\gamma$ are defined in \eqref{Fellbar}.
Since $\langle -\ell w_0\mu, \alpha^\vee\rangle\in \ell\ZZ$ then
\begin{align*}
N_{\gamma-\ell w_0\mu} 
= \Card&\{\alpha\in R^+\ |\ \langle \gamma-\ell w_0\mu+\rho, \alpha^\vee\rangle \in \ell\ZZ\} \\
&= \Card\{\alpha\in R^+\ |\ \langle \gamma+\rho, \alpha^\vee\rangle \in \ell\ZZ\} 
=N_{\gamma}.
\end{align*}
Thus
\begin{align*}
\overline{X^\mu\cdot \vert \gamma\rangle }
&= \overline{\vert \gamma -\ell w_0\mu \rangle } 
= (-1)^{\ell(w_0)}
(t^{-\frac12})^{\ell(w_0)-N_{\gamma-\ell w_0\mu}} 
\vert w_0\circ(\gamma-\ell w_0\mu)\rangle \\
&= (-1)^{\ell(w_0)}
(t^{-\frac12})^{\ell(w_0)-N_{\gamma-\ell w_0\mu}} 
\vert w_0(\gamma+\rho)-\rho-\ell\mu \rangle \\
&= (-1)^{\ell(w_0)}
(t^{-\frac12})^{\ell(w_0)-N_{\gamma-\ell w_0\mu}} 
\vert w_0\circ \gamma-\ell\mu \rangle \\
&= (-1)^{\ell(w_0)}
(t^{-\frac12})^{\ell(w_0)-N_{\gamma}} 
\vert w_0\circ \gamma-\ell \mu \rangle \\
&= (-1)^{\ell(w_0)}
(t^{-\frac12})^{\ell(w_0)-N_{\gamma}} 
X^{w_0\mu}\cdot \vert w_0\circ \gamma\rangle 
= X^{w_0\mu}\cdot \overline{ \vert \gamma\rangle},
\end{align*}
and since $s_{\lambda_1^*}$ is $W_0$-invariant,
$$\overline{ s_{\lambda_1^*} \cdot C_{\lambda_0} }
= (w_0 s_{\lambda_1^*})\cdot \overline{C_{\lambda_0}}
=s_{\lambda_1^*} \cdot C_{\lambda_0}.$$

\smallskip\noindent
(b)
Let $\lambda=\lambda_0+\ell\lambda_1$ as in the statement of the Theorem and let
$a \equiv b$ mean $a=b \bmod t^{\frac12}$.
By the second formula in \eqref{Clambdadefn},
\begin{equation}
s_{\lambda_1^*}\cdot C_{\lambda_0} \equiv s_{\lambda_1^*}\cdot \vert \lambda_0 \rangle 
= \sum_{p\in B(\lambda_1^*)} X^{\mathrm{wt}(p)}  \vert \lambda_0 \rangle
=\sum_{p\in B(\lambda_1^*)} \vert \ell\mathrm{wt}(p)^*+\lambda_0 \rangle. 
\label{SCexpansion}
\end{equation}
By \eqref{Fstraightening},
if $\lambda\in\mathfrak{a}_{\mathbb{Z}}^*$ and $\langle\lambda+\rho,\alpha_i^\vee\rangle\geq 0$ 
then
$$
\vert s_i\circ\lambda \rangle
\equiv
\begin{cases}
- \vert \lambda \rangle, 
&\hbox{if $\langle\lambda+\rho,\alpha_i^\vee\rangle \in \ell\ZZ_{\geq 0}$,}\\
0, &\hbox{if $0 < \langle \lambda+\rho, \alpha_i^\vee \rangle < \ell$,} \\
- \vert \lambda^{(1)} \rangle, &\hbox{otherwise,}
\end{cases}
$$
where $\lambda^{(1)}=\lambda-j\alpha_i$ if $\langle\lambda+\rho,\alpha_i^\vee\rangle=k\ell+j$ with $j\in\{0,1,\ldots, \ell-1\}$.  Since $\lambda^{(1)} = \lambda$ if $\langle \lambda+\rho, \alpha_i^\vee\rangle\in \ell\ZZ_{\ge 0}$, the first case can be viewed as a special case of the last case
to read
$$
\vert s_i\circ\lambda \rangle
\equiv
\begin{cases}
0, &\hbox{if $0 < \langle \lambda+\rho, \alpha_i^\vee \rangle < \ell$,} \\
- \vert \lambda^{(1)} \rangle, &\hbox{otherwise.}
\end{cases}
$$

Assume $\langle \nu+\rho, \alpha_i^\vee\rangle\in \ZZ_{\le0}$ and let
$\lambda = s_i\circ(\lambda_0 + \ell \nu)$. 
Since $\langle \rho, \alpha_i^\vee\rangle = 1$ then
\begin{align*}
\langle \lambda+\rho, \alpha_i^\vee\rangle
&= \langle s_i\circ(\lambda_0+\ell\nu)+\rho, \alpha_i^\vee\rangle
=\langle s_i(\lambda_0+\ell\nu+\rho), \alpha_i^\vee\rangle 
=\langle \lambda_0+\ell\nu+\rho, s_i\alpha_i^\vee\rangle \\
&= - \langle \lambda_0+\ell\nu+\rho, \alpha_i^\vee\rangle 
= \ell(-\langle \nu+\rho, \alpha_i^\vee\rangle) + (\ell -1 - \langle \lambda_0, \alpha_i^\vee\rangle).
\end{align*}
Since $\lambda_0\in \Pi_\ell$ then $0\le \ell-1-\langle \lambda_0,\alpha_i^\vee\rangle <\ell$
and so
\begin{align*}
\lambda^{(1)} 
&= \lambda 
- (\ell-1-\langle \lambda_0, \alpha_i^\vee\rangle)\alpha_i  
= s_i\circ(\lambda_0 + \ell \nu)
- (\ell-1-\langle \lambda_0, \alpha_i^\vee\rangle)\alpha_i \\
&= s_i\lambda_0 + \ell s_i\nu+s_i\rho-\rho
- (\ell-1-\langle \lambda_0, \alpha_i^\vee\rangle)\alpha_i \\
&= (\lambda_0 - \langle \lambda_0, \alpha_i^\vee\rangle \alpha_i)
+\ell(s_i\nu+s_i\rho-\rho)+(\ell-1)\alpha_i 
- (\ell-1-\langle \lambda_0, \alpha_i^\vee\rangle)\alpha_i \\
&= \lambda_0 + \ell(s_i\circ \nu).
\end{align*}
Thus, since $s_i\circ\lambda = \lambda_0+\ell\nu$,
$$
\vert \lambda_0+\ell\nu \rangle
\equiv
\begin{cases}
0, &\hbox{if $\langle \nu+\rho, \alpha_i^\vee\rangle = 0$,} \\
- \vert \lambda_0+\ell(s_i\circ \nu)\rangle, 
&\hbox{if $\langle \nu+\rho, \alpha_i^\vee\rangle < 0$.} \\
\end{cases}
$$
Since $s_i\circ \nu = \nu$ when $\langle \nu+\rho, \alpha_i^\vee\rangle=0$, then
$\vert \lambda_0+\ell\nu \rangle
\equiv
- \vert \lambda_0+\ell (s_i\circ \nu)\rangle$ when
$\langle \nu+\rho, \alpha_i^\vee \rangle\in \ZZ_{\le 0}$ and,
replacing $\nu$ by $s_i\circ \nu$, gives 
$\vert \lambda_0+\ell\nu \rangle\equiv
- \vert \lambda_0+\ell(s_i\circ \nu)\rangle$ 
for $\langle \nu+\rho, \alpha_i^\vee \rangle\in \ZZ_{\ge 0}$. 
Thus
\begin{equation}
\vert \lambda_0+\ell\nu \rangle
\equiv
- \vert \lambda_0+\ell(s_i\circ \nu)\rangle,
\qquad\hbox{for $\nu\in \fa_\ZZ$.}
\label{modtstraightening}
\end{equation}

With formula \eqref{modtstraightening} established, follow 
\cite[proof of Theorem 5.5]{Ra} (see also \cite[proof of Theorem 9.1]{Li})
to define an involution $\iota$ on the set $B(\lambda_1^*)\setminus\{p_{\lambda_1^*}^+\}$,
where $p_{\lambda_1^*}^+$ is the unique highest weight path in $B(\lambda_1^*)$.

\medskip
Let $p\in B(\lambda_1^*)$ and $p\neq p_{\lambda_1^*}^+$. 
Since $p\neq p_{\lambda_1^*}^+$ the path $p$ crosses a wall out of the fundamental chamber
at some point during its trajectory. 
Let $r$ be such that the first time $p$ leaves the dominant chamber is by crossing the 
hyperplane $\{ x\in \fa_\RR^*\ |\ \langle x, \alpha_r^\vee\rangle = 0\}$. 
Letting $\tilde e_r$ and $\tilde f_r$ denote the root operators on $B(\lambda_1^*)$,
the $r$-string containing $p$ is 
$$S_r(p) =\{q\in B(\lambda_1^*)\ |\ 
\hbox{$q=\tilde{e}_r^kp$ or $q=\tilde{f}_r^k p$ where $k\in \ZZ_{\geq 0}$}\}.$$
  Let
$$
\hbox{$\iota(p)$ be the element of $S_r(p)$ such that 
$\mathrm{wt}(\iota(p)) = s_r\circ \mathrm{wt}(p)$.}
$$
By \eqref{modtstraightening},
$$
\vert \lambda_0+\ell\wt(p)^*\rangle
\equiv 
- \vert \lambda_0+\ell (s_r^*\circ \wt(p)^*) \rangle
= - \vert \lambda_0+\ell(s_r\circ \wt(p))^*\rangle
= - \vert \lambda_0+\ell\wt(\iota(p))^*\rangle,
$$
and so the map $\iota$ partitions the set $B(\lambda_1^*)\setminus\{p_{\lambda_1^*}^+\}$ 
into pairs $\{p, \iota(p)\}$ which cancel each other in the mod $t^{\frac{1}{2}}$
straightening of the terms of 
$s_{\lambda_1^*}\cdot C_{\lambda_0}$ in \eqref{SCexpansion}.
Thus 
$$s_{\lambda_1^*}\cdot C_{\lambda_0} \equiv 
\vert \ell(\lambda_1^*)^*+\lambda_0\rangle
=\vert \lambda_0+ \ell \lambda_1\rangle,
\qquad\hbox{which proves (b).}
$$
\end{proof}

\section{The Casselman-Shalika formula}

In order to establish the Casselman-Shalika formula it is necessary to use the connection
between the abstract Fock space $\cF_\ell$ and the affine Hecke algebra $H$.  Let us recall
this relationship from \cite{LRS}.

\subsection{The affine Hecke algebra $H$}

Keep the notation for the finite Weyl group $W_0$, the simple reflections $s_1, \ldots, s_n$ 
and the weight lattice $\fa_\ZZ^*$ as in \eqref{wtsWeylgpdefn}.  
For $i,j\in \{1, \ldots, n\}$ with $i\ne j$, let
$m_{ij}$ denote the order of $s_is_j$ in $W_0$
so that $s_i^2=1$ and $(s_is_j)^{m_{ij}}=1$ are the relations for the Coxeter presentation of $W_0$.
Let $\KK = \ZZ[t^{\frac12}, t^{-\frac12}]$.  The \emph{affine Hecke algebra} is
\begin{equation}
H = \hbox{$\KK$-span}\{X^\mu T_w\ |\ \mu\in \fa_\ZZ^*, w\in W_0\},
\label{affHeckedefn}
\end{equation}
with $\KK$-basis $\{X^\mu T_w\ |\ \mu\in \fa_\ZZ^*, w\in W_0\}$ and relations
\begin{equation}
(T_{s_i} - t^{\frac{1}{2}})(T_{s_i} + t^{-\frac{1}{2}}) =0, \qquad
\underbrace{T_{s_i}T_{s_j}T_{s_i}\ldots}_{m_{ij}\ \mathrm{factors}} 
= \underbrace{T_{s_j}T_{s_i}T_{s_j}\ldots}_{m_{ij}\ \mathrm{factors}},
\end{equation}
\begin{equation}
X^{\lambda + \mu} = X^{\lambda}X^{\mu}, \quad\hbox{and}\quad
T_{s_i}X^{\lambda} - X^{s_i\lambda}T_{s_i} = (t^{\frac{1}{2}} - t^{-\frac{1}{2}})\left(\frac{X^{\lambda} - X^{s_i\lambda}}{1 - X^{-\alpha_i}}\right),
\label{Hrels}
\end{equation}
for $i,j\in \{1,\ldots, n\}$ with $i\ne j$ and $\lambda, \mu\in \fa_\ZZ^*$.
The \emph{bar involution on $H$} is the $\ZZ$-linear automorphism 
$\overline{\phantom{T}}\colon H\to H$ given by 
\begin{equation}
\overline{t^{\frac12}} = t^{-\frac12}, \qquad \overline{T_{s_i}} = T_{s_i}^{-1},
\qquad\hbox{and}\qquad
\overline{X^\lambda} = T_{w_0}X^{w_0\lambda}T_{w_0}^{-1}.
\label{baronH}
\end{equation}
for $i=1,\ldots, n$ and
$\lambda, \mu\in  \fa_\ZZ^*$.
For $\mu\in \fa_\ZZ^*$ and $w\in W_0$ define 
\begin{equation}
X^{t_\mu w} = X^\mu (T_{w^{-1}})^{-1}
\qquad\hbox{and}\qquad
T_{t_\mu w} 
=T_x X^{\mu^+}T_{w_{\mu^+}}(T_{w^{-1}xw_{\mu^+}})^{-1},
\label{BernsteinCoxeterconversion}
\end{equation}
where $\mu^+$ is the dominant representative of $W_0\mu$, $x\in W_0$ is minimal length such that
$\mu = x\mu^+$ and $w_{\mu^+}$ is the longest element of the stabilizer $W_{\mu^+} = \mathrm{Stab}_{W_0}(\mu+)$.
Define
$$\varepsilon_0 
= (-t^{\frac12})^{\ell(w_0)}\sum_{z\in W_0} (-t^{-\frac12})^{\ell(z)} T_z
\qquad\hbox{and}\qquad
\mathbf{1}_0 
= (t^{-\frac12})^{\ell(w_\nu)}\sum_{z\in W_0} (t^{\frac12})^{\ell(z)} T_z,
$$
so that
\begin{equation}
\overline{\varepsilon_0} = \varepsilon_0,
\quad
\overline{\mathbf{1}_0} = \mathbf{1}_0,
\qquad\hbox{and}\qquad
\varepsilon_0 T_{s_i} = -t^{-\frac12}\varepsilon_0,
\quad\hbox{and}\quad
T_{s_i}\mathbf{1}_0  = t^{\frac12}\mathbf{1}_0,
\label{e0defn}
\end{equation}
for $i\in \{1, \ldots, n\}$.
The algebra $\KK[X]$ defined in \eqref{KXdefn} is a subalgebra of $H$ and, by a theorem of
Bernstein (see \cite[Theorem 1.4]{NR}), the center of $H$ is the ring of symmetric functions,
\begin{equation}
Z(H) = \KK[X]^{W_0}.
\label{centerofH}
\end{equation}

\begin{remark}
Formulas \eqref{baronH} and \eqref{BernsteinCoxeterconversion} are just a reformulation of the usual 
bar involution 
and the conversion 
between the Bernstein and Coxeter presentations of the affine Hecke algebra
(see for example \cite[Lemma 2.8 and (1.22)]{NR}).
\end{remark}

\subsection{The relation between $H$ and the abstract Fock space $\cF_\ell$}

In this subsection we follow \cite[\S4.2]{LRS}.
The \emph{affine Weyl group} is
\begin{equation}
W = \{ t_\mu w\ |\ \mu\in \fa_\ZZ^*, w\in W_0\}, \qquad\hbox{with}\qquad
t_\mu t_\nu = t_{\mu+\nu}, \quad\hbox{and}\quad
wt_\mu = t_{w\mu} w,
\label{affWeyldefn}
\end{equation}
for $\mu, \nu\in \fa_\ZZ^*$ and $w\in W_0$.
Let $\varphi^\vee$ and $h$ be as in \eqref{Coxeternumber}.
For $\ell\in \ZZ_{>0}$,
the \emph{level $(-\ell-h)$ dot action of $W$ on $\fa_\ZZ^*$} is given by
\begin{equation}
(t_\mu w)\circ\lambda = (w\circ\lambda) - \ell \mu = w(\lambda+\rho) - \rho - \ell\mu,
\label{levelelldotaction}
\end{equation}
for $\mu\in \fa_\ZZ^*$, $w\in W_0$ and $\lambda\in \fa_\ZZ^*$.  Note that this is
an extension of the dot action of $W_0$ given in \eqref{dotaction}.
Define
\begin{equation}
A_{-\ell-h} = \{ \nu\in \fa^*_\ZZ\ |\ 
\hbox{$\langle \nu,\varphi^\vee\rangle \ge -\ell-1$ and
$\langle \nu, \alpha_i^\vee\rangle \le -1$ for $i\in \{1, \ldots, n\}$}\}.
\label{ellAlcove}
\end{equation}
and
\begin{equation}
\cP^+_{-\ell-h} = \bigoplus_{\nu\in A_{-\ell-h}} \varepsilon_0 H\mathbf{p}_\nu,
\label{Heckemoduledefn}
\end{equation}
where $\varepsilon_0$ is as in \eqref{e0defn} and $\mathbf{p}_\nu$ are formal symbols 
indexed by $\nu\in A_{-\ell-h}$ satisfying
$$\overline{\mathbf{p}_\nu} = \mathbf{p}_\nu
\qquad\hbox{and}\qquad
\hbox{$T_y\mathbf{p}_\nu = (t^{\frac12})^{\ell(y)}\mathbf{p}_\nu$ for $y\in W_\nu$,}
$$
where $W_\nu = \mathrm{Stab}_W(\nu)$ is the stabilizer of $\nu$ 
under the level $(-\ell-h)$ dot action of $W$ on $\fa_\ZZ^*$.
Define a bar involution 
\begin{equation}
\overline{\phantom{T}}\colon \cP_{-\ell-h}^+\to \cP_{-\ell-h}^+
\qquad\hbox{by}\qquad
\overline{\varepsilon_0 f \mathbf{p}_\nu } = \varepsilon_0 \bar f \mathbf{p}_\nu,
\quad\hbox{for $\nu\in A_{-\ell-h}$ and $f\in H$.}
\label{barforP}
\end{equation}
For $\lambda\in \fa^*_\ZZ$ define
\begin{equation}
[X_\lambda] = [X_{w_0v\circ\nu}]
= \varepsilon_0 X^v \mathbf{p}_\nu,
\qquad\hbox{where}\quad
\lambda = w_0v\circ\nu \quad\hbox{with $\nu\in A_{-\ell-h}$,}
\label{bracketTXdefn}
\end{equation}
and $v\in W$ is such that $X^{vu} = X^{v}T_{u}$ for any $u\in W_{\nu}$.  It is helpful
to stress that the $(-\ell-h)$ dot action of \eqref{levelelldotaction} applies here so that, when
$v = t_\mu w$ with $\mu\in \fa_\ZZ^*$ and $w\in W_0$, 
then $\lambda = -\ell w_0\mu+(w_0w)\circ \nu$ and
\begin{equation}
[X_\lambda] = [X_{-\ell w_0\mu+(w_0w)\circ \nu}] = \varepsilon_0 X^{t_\mu w}\mathbf{p}_\nu
= \varepsilon_0 X^\mu(T_{w^{-1}})^{-1}\mathbf{p}_\nu.
\label{bracketXexpanded}
\end{equation}

With these notations, a main result of \cite{LRS} is 
\begin{thm} (see \cite[Theorem 4.7]{LRS}) \label{abstracttoHecke} 
Let $\le$ be the dominance order on the set $(\fa_\ZZ^*)^+$ of
dominant integral weights.
Then the $\KK$-linear map $\Phi\colon \cF_\ell\to \cP^+_{-\ell-h}$ given by
\begin{equation}
\Phi(\ \vert\lambda\rangle\ ) = [X_\lambda],
\qquad \hbox{for $\lambda\in \fa_\ZZ^*$},
\label{Phidefn}
\end{equation}
is a well defined $\KK$-module isomorphism satisfying 
$\overline{\Phi(f)} = \Phi(\overline{f})$.
\end{thm}

\noindent
Since elements of $Z(H)=\KK[X]^{W_0}$ commute with $\varepsilon_0$ there
is a $\KK[X]^{W_0}$-action on $\cP_{-\ell-h}$ by left multiplication.  The pullback of
this action by the isomorphism $\Phi$ is the source of the $\KK[X]^{W_0}$
action on $\cF_\ell$ given in Proposition \ref{KXaction},
\begin{equation}
z\Phi(f) = \Phi(zf),
\qquad\hbox{for $z\in Z(H) = \KK[X]^{W_0}$ and $f\in \cF_\ell$.}
\label{Phicommute}
\end{equation}

\subsection{Deducing the Casselman-Shalika formula}

For $\mu\in \fa_\ZZ^*$ define the ``Whittaker function'' $A_\mu\in  \varepsilon_0 H \mathbf{1}_0$ by
\begin{equation}
A_\mu = \varepsilon_0 X^\mu \mathbf{1}_0.
\label{Amudefn}
\end{equation}
See, for example, \cite[\S6]{HKP} for the connection between $p$-adic groups and 
the affine Hecke algebra and the 
explanation of why $A_\mu$ is equivalent to the data of a (spherical) Whittaker function
for a $p$-adic group.
As proved carefully in \cite[Theorem 2.7]{NR}, it follows from \eqref{e0defn} and \eqref{Hrels} that
$$
\hbox{$\varepsilon_0 H \mathbf{1}_0$ has $\KK$-basis}\quad
\left\{ A_{\lambda+\rho}\ |\ \hbox{
$\langle \lambda+\rho, \alpha_i\rangle\in \ZZ_{\ge 0}$
for $i\in \{1, \ldots, n\}$}\right\}.
$$ 
Following \cite[Theorem 2.4]{NR}, the Satake isomorphism, 
$\KK[X]^{W_0}\cong \mathbf{1}_0H\mathbf{1}_0$, and 
the Casselman-Shalika formula, $A_{\lambda+\rho}=s_\lambda A_\rho$,
can be formulated by the following diagram of vector space (free $\KK$-module) isomorphisms:
\begin{equation}
\begin{matrix}
Z(H) = \KK[X]^{W_0} &\stackrel{\sim}{\longrightarrow} &\mathbf{1}_0 H \mathbf{1}_0
&\stackrel{\sim}{\longrightarrow} &\varepsilon_0 H \mathbf{1}_0 \\
f &\longmapsto &f\mathbf{1}_0 &\longmapsto &A_\rho f \mathbf{1}_0 \\
s_{\lambda} &\longmapsto &s_{\lambda}\mathbf{1}_0 &\longmapsto &A_{\lambda+\rho}\end{matrix}
\label{LuszCasShi}
\end{equation}
This diagram has particular importance due to the fact that $\KK[X]^{W_0}$ is an avatar of the 
Grothendieck group of the category $\mathrm{Rep}(G)$ 
of finite dimensional representations of $G$,
the spherical Hecke algebra $\mathbf{1}_0 H \mathbf{1}_0$ is a form of the Grothendieck
group of $K$-equivariant perverse sheaves on the loop Grassmanian $Gr$, 
and $\varepsilon_0 H\mathbf{1}_0$
is isomorphic to the Grothendieck group of Whittaker sheaves (appropriately formulated
$N$-equivariant sheaves
on $Gr$), see \cite{FGV}.

Our proof of the Casselman-Shalika formula is accomplished by restricting Theorem \ref{TenPrdThm}
to the summand in \eqref{Heckemoduledefn} corresponding to $-\rho\in A_{-\ell-h}$.
We shall identify this summand with $\varepsilon_0 H \mathbf{1}_0$ via the $Z(H)$-isomorphism
$$
\begin{matrix}
\varepsilon_0 H \mathbf{1}_0 &\stackrel{\sim}{\longrightarrow} &\varepsilon_0 H \mathbf{p}_{-\rho} \\
\varepsilon_0 X^\mu \mathbf{1}_0 &\longmapsto &\varepsilon_0 X^\mu \mathbf{p}_{-\rho}
\end{matrix}
$$
Using the level $(-\ell-h)$ dot action of $W$ from \eqref{levelelldotaction},
the stabilizer of $-\rho$ is $W_0$ and 
$$
W\circ (-\rho) = \{ t_{-\lambda}\circ(-\rho)\ |\ \lambda\in \fa_\ZZ^*\} 
= \{ \ell\lambda-\rho\ |\ \lambda\in \fa_\ZZ^*\}.
$$
Since
$\langle (\ell\lambda-\rho)+\rho, \alpha^\vee\rangle \in \ell\ZZ$ for $\alpha\in R^+$,
the straightening law \eqref{Fstraightening} for elements of $W\circ (-\rho)$ is
\begin{equation}
\vert s_i\circ (\ell\lambda-\rho)\rangle = - \vert \ell\lambda-\rho\rangle.
\label{negrhostraightening}
\end{equation}

\begin{thm} (Casselman-Shalika)  For $\lambda\in (\fa_\ZZ^*)^+$ and $\mu\in \fa_\ZZ^*$ let $s_\lambda$ be the Weyl character as defined in \eqref{Weylchdefn} and let $A_\mu$ be the Whittaker function as defined
in \eqref{Amudefn}. Then
$$s_\lambda A_\rho = A_{\lambda+\rho}.$$
\end{thm}
\begin{proof}
Using \eqref{negrhostraightening},
$$\overline{\vert \ell\lambda-\rho \rangle} 
= (-1)^{\ell(w_0)}(t^{-\frac12})^{\ell(w_0)-\ell(w_0)}
\vert w_0\circ (\ell\lambda-\rho)\rangle 
= (-1)^{\ell(w_0)} \vert w_0\circ (\ell \lambda-\rho) \rangle
= \vert \ell\lambda -\rho \rangle
$$
and thus $\vert \ell\lambda-\rho\rangle$ satisfies the conditions of \eqref{Clambdadefn} so that
\begin{equation}
C_{\ell\lambda-\rho} = \vert \ell\lambda-\rho \rangle,
\qquad\hbox{for $\lambda\in (\fa_\ZZ^*)^+$.}
\label{KLisket}
\end{equation}
By \eqref{bracketXexpanded} and \eqref{Amudefn},
$$[X_{-\ell w_0\mu-\rho}] 
= \varepsilon_0 X^\mu T_{w_0}^{-1}\mathbf{p}_{-\rho} 
= t^{-\ell(w_0)/2} \varepsilon_0 X^\mu \mathbf{p}_{-\rho} 
= t^{-\ell(w_0)/2}A_\mu,
\qquad\hbox{for $\mu\in (\fa_\ZZ^*)^+$}.$$
Using \eqref{Phicommute}, \eqref{KLisket}, \eqref{Phidefn}
and that $w_0\rho = -\rho$,
\begin{align*}
t^{-\ell(w_0)/2} s_\lambda &A_{\rho} 
= s_\lambda [X_{-\ell w_0\rho - \rho}] 
= s_\lambda [X_{(\ell-1)\rho}] 
= s_\lambda \Phi(\vert (\ell-1)\rho\rangle) 
= \Phi(s_\lambda \, \vert (\ell-1)\rho\rangle\ ) \\
&= \Phi(s_{\lambda} C_{(\ell-1)\rho}) = \Phi(C_{-\ell w_0\lambda+(\ell-1)\rho}),
\qquad\hbox{by Theorem \ref{TenPrdThm},} \\
&= \Phi(\vert (-\ell w_0\lambda)+(\ell-1)\rho \rangle) 
= [X_{-\ell w_0\lambda+(\ell-1)\rho}]
= [X_{-\ell w_0(\lambda+\rho)-\rho}] \\
&
= t^{-\ell(w_0)/2} A_{\lambda+\rho}.
\end{align*}
\end{proof}

\section{Quantum groups and LLT polynomials}

In this section we describe the main motivation for Theorem \ref{TenPrdThm} namely,
the Steinberg-Lusztig tensor product theorem for representations of quantum groups 
at roots of unity.  Then we explain the connection between these results and the theory of 
LLT polynomials.

\subsection{Representations of quantum groups at a root of unity}

Let $\mathring{\fg}$ be the Lie algebra of the group $G$ alluded to in \eqref{wtsWeylgpdefn}.
Let $q\in \CC^\times$ and let $U_q(\mathring{\fg})$ be the Drinfel'd-Jimbo quantum group corresponding to $\mathring{\fg}$.  Let
\begin{align*}
&\Delta_q(\lambda) &\hbox{the Weyl module for $U_q(\mathring{\fg})$ of highest weight $\lambda$,} \\
&L_q(\lambda) &\hbox{the simple module for $U_q(\mathring{\fg})$ of highest weight $\lambda$,}
\end{align*}
Let 
$$\hbox{$K(\hbox{fd$U_q(\mathring{\fg})$-mod})$ 
be the free $\ZZ[t^{\frac12},t^{-\frac12}]$-module generated by symbols
$[\Delta_q(\lambda)]$,}
$$ 
for $\lambda\in \fa_\ZZ^*$. For $\mu\in\fa_\ZZ^*$, denote by $W^\mu$, resp. ${}^\mu W$, the set of minimal length coset representatives for $W/W_\mu$, resp. $W_\mu\backslash W$.
Define elements $[L_q(w_0 y\circ\nu)]$, 
for $\nu\in A_{-\ell-h}$ and $y\in {}^0 W$ such that $w_0 y\in W^\nu$, by the equation
$$[\Delta_q(w_0 x\circ\nu)] = 
\sum_{y\le x} 
\left(\sum_{i\in \ZZ_{\ge 0}} \left[ \frac{\Delta_q(w_0 x\circ\nu)^{(i)}}
{\Delta_q(w_0 x\circ\nu)^{(i+1)} } :
L_q(w_0 y\circ\nu)\right] (t^{\frac12})^i\right) [L_q(w_0 y\circ\nu)],$$
where $[M:L_q(\mu)]$ denotes the multiplicity of the simple $\fg$-module 
$L_q(\mu)$ of highest weight
$\mu$ in a composition series of $M$ and
$$\Delta_q(\lambda)=\Delta_q(\lambda)^{(0)} 
\supseteq \Delta_q(\lambda)^{(1)}\supseteq
\cdots
\qquad\hbox{is the Jantzen filtration of $\Delta_q(\lambda)$}
$$
(see, for example, \cite[\S1.4, \S2.3 and \S2.10 and Cor.\ 2.14]{Sh}  and \cite[\S4]{JM} for the Jantzen filtration in this context).

The combination of \cite[(3.20)]{LRS}
and \cite[Theorem 4.7]{LRS}) is the following connection between the representation theory of
the quantum group at a root of unity and the abstract Fock space.

\begin{thm} \label{qgrouptoFockspace}  
Let $\ell\in \ZZ_{>0}$ and let $q\in \CC^\times$ such that $q^{2\ell}=1$.\hfil\break 
Let
$\KK = \ZZ[t^{\frac12}, t^{-\frac12}]$.  Then the $\KK$-linear map given by
$$\begin{matrix}
K(\hbox{\emph{fd$U_q(\mathring{\fg})$-mod}}) &\stackrel{\Psi_2}\longrightarrow & \cF_\ell \\
[\Delta_q(\lambda)] &\longmapsto &\vert \lambda \rangle \\
[L_q(\lambda)] &\longmapsto &C_\lambda 
\end{matrix}
\qquad\hbox{is a well defined isomorphism of $\ZZ[t^{\frac12}, t^{-\frac12}]$-modules.}
$$
\end{thm}

The enveloping algebra $U\mathring{\fg}$ has a presentation by generators 
$e_1,\ldots, e_n$, $f_1, \ldots, f_n$ and $h_1, \ldots, h_n$ 
and Serre relations and the quantum group $U_q\mathring{\fg}$ 
has a presentation by generators $E_1, \ldots, E_n$, $F_1, \ldots, F_n$, and $K_1, \ldots, K_n$
and quantum Serre relations such that, at $q=1$, 
$E_i$ becomes $e_i$ and $F_i$ becomes $f_i$.
Following \cite{Lu89} and \cite[Theorem 9.3.12]{CP},
with appropriate restrictions on $\ell$ as in \cite[just before Proposition 9.3.5 and Theorem 9.3.12]{CP},  
the \emph{Frobenius map} is the Hopf algebra homomorphism
\begin{equation}
\begin{matrix} Fr \colon &U_q\mathring{\fg} &\longrightarrow &\ \ U\mathring{\fg}\hfill \\
&E_i^{(r)} &\mapsto
&\begin{cases} e_i^{(r/\ell)}, &\hbox{if $\ell$ divides $r$,} \\
0, &\hbox{otherwise,} \\
\end{cases} \\
&F_i^{(r)} &\mapsto
&\begin{cases} f_i^{(r/\ell)}, &\hbox{if $\ell$ divides $r$,} \\
0, &\hbox{otherwise,} \\
\end{cases} \\
&K_i &\mapsto &\ \ 1. \hfill 
\end{matrix}
\end{equation}
The \emph{Frobenius twist} of a $U\mathring{\fg}$-module $M$ 
is the $U_q\mathring{\fg}$-module $M^{Fr}$ 
with underlying vector space $M$ and $U_q\mathring{\fg}$-action given by
$$um = Fr(u)m, \qquad\hbox{for $u\in U_q\mathring{\fg}$ and $m\in M$.} $$

\begin{thm} \label{SLThm}  (\cite[Theorem 7.4]{Lu89}, see also \cite[11.2.9]{CP})
Let $\ell\in \ZZ_{>0}$ and let $\Pi_\ell$ be as defined in \eqref{Pielldefn}.  
Let $\lambda\in (\fa_\ZZ^*)^+$ and write 
$$\lambda = \ell\lambda_1+\lambda_0,
\qquad\hbox{with $\lambda_0\in \Pi_\ell$ and $\lambda_1\in (\fa_\ZZ^*)^+$.} 
$$
Let $q\in \CC^\times$ be such that $q^{2\ell}=1$ and let $L_q(\lambda)$ denote the 
simple $U_q\mathring{\fg}$-module of highest weight $\lambda$.
Then
$$
L_q(\lambda) \cong \Delta(\lambda_1)^{Fr}\otimes L_q(\lambda_0),$$
where $\Delta(\mu)$ denotes the irreducible $U\mathring{\fg}$-module of highest weight $\mu$.
\end{thm}

\noindent
Accepting Theorem \ref{qgrouptoFockspace}, Theorem \ref{SLThm} is equivalent to the product theorem for abstract Fock space,
Theorem \ref{TenPrdThm}.

\subsection{LLT polynomials for general Lie type}

In \cite{LLT} and \cite[(43)]{LT} and \cite[Definition 6.6]{GH}, 
the LLT polynomials for type A are defined by
\begin{equation}
G^{(\ell)}_{\mu/\nu}(x, t^{-1}) 
\sum_{T\in SSRT_\ell(\mu/\nu)} t^{-\mathrm{spin}(T)}x^T,
\label{LLTdefn}
\end{equation}
where $SSRT_\ell(\mu/\nu)$ is the set of semistandard $\ell$ ribbon tableaux of
shape $\mu/\nu$, $\mathrm{spin}(T)$ is the spin of the tableaux $T$ and 
$X^T$ is the weight of the tableaux $T$ (see \cite[\S6.5]{GH} for an efficient review of the
combinatorial definitions of semistandard ribbon tableaux, spin and weight).

In Lecouvey \cite{Lcy}, there is a definition of LLT polynomials for general Lie type 
generalizing the definition of \cite{LLT} from type A which proceeds as follows.
Define a $\KK$-algebra homomorphism
$$
\begin{matrix}
\psi_\ell\colon &\KK[X] &\longrightarrow &\KK[X] \\
&X^\mu &\longmapsto &X^{\ell\mu}
\end{matrix}
\quad\hbox{so that}\quad
\psi_\ell(s_\lambda) = \mathrm{char}(\Delta_q(\lambda)^{Fr}),
$$
in the framework of Theorem \ref{SLThm}.  Then \cite[(57)]{Lcy} defines
\begin{equation}
G_\mu^\ell
= \sum_{\lambda\in (\fa_\ZZ^*)^+} p_{\ell\lambda, \mu}  s_\lambda,
\label{Lecouveydefn}
\end{equation}
where $p_{\ell\lambda,\mu}\in \ZZ[t^{\frac12}]$ are as in \eqref{Clambdadefn}.  
As pointed out in \cite[Cor.\ 5.1.3]{Lcy},  Theorem \ref{qgrouptoFockspace} gives
\begin{equation*}
\psi_\ell(s_\lambda)
= \mathrm{char}(\Delta(\lambda)^{Fr}) 
= \mathrm{char}(L_q(\ell\lambda)) 
= \sum_{\mu\in (\fa_\ZZ^*)^+} p_{\ell\lambda,\mu}(1) \mathrm{char}(\Delta_q(\mu))
= \sum_{\mu\in (\fa_\ZZ^*)^+} p_{\ell\lambda,\mu}(1) s_\mu.
\end{equation*}
As explained carefully in 
\cite[Theorem 4.8(b)]{LRS}, the polynomials $p_{\ell\lambda,\mu}$ are
parabolic singular Kazhdan-Lusztig polynomials.

In \cite[Definition 5.12 and Corollary 6.4]{GH} there is another definition of LLT polynomials for general 
Lie type:
\begin{equation}
\cL^G_{L, \beta,\gamma}
= t^{l_{\beta-\gamma}+\ell(w)-\ell(v)}
\sum_{\lambda\in (\fa_\ZZ^*)^+} Q_{\mu\nu}^\lambda s_\lambda,
\quad\hbox{where 
}\quad
s_{\lambda^*} \cdot \vert\nu\rangle = \sum_{\mu} Q_{\mu\nu}^\lambda \vert\mu\rangle
\label{GHdefn}
\end{equation}
determine the polynomials $Q_{\nu\mu}^\lambda$.  Here $G$ is the 
reductive algebraic group alluded to in \eqref{wtsWeylgpdefn}, 
$L$ is a Levi subgroup of $G$ with Weyl group $W_\nu$,
$l_{\beta-\gamma}$ is the nonnegative integer defined in \cite[Remark 5.10]{GH},
and 
$$
\mu=v\circ (\eta+\ell\beta) \quad\hbox{and}\quad
\nu = w\circ (\eta+k\gamma),
\quad\hbox{where}\quad
v\in W_0 t_{\beta} W_\eta
\quad\hbox{and}\quad 
w\in W_0 t_\gamma W_\eta
$$
are minimal representatives.  At this point, the reader's discomfort
occurring from the transitions between $\beta$ and $\gamma$ and $v$ and $w$ 
and $\mu$ and $\nu$ is mitigated by recognizing that the relation between these two 
definitions occurs in the special case $\nu=0$: Theorem \ref{TenPrdThm}
and \eqref{KLisket} and the definition of $Q_{\mu\nu}^\lambda$ in \eqref{GHdefn} give
$$C_{\ell\lambda} = s_{\lambda^*}\cdot C_0
= s_{\lambda^*}\cdot \vert 0\rangle = \sum_\mu Q_{\mu0}^\lambda \vert \mu\rangle,
\quad\hbox{and comparing with \eqref{Clambdadefn} gives}\quad
p_{\ell\lambda,\mu} = Q_{\mu,0}^\lambda$$
and specifies the close relationship between $G_\mu^\ell$ and $\cL_{L, \beta, \gamma}^G$ 
which occurs at $\nu=0$.  They are the same up to a power of $t$.

\section{Tensor product theorem on affine Lie algebra representations}

Let $\mathring{\fg}$ be the Lie algebra of $G$ and let
$\fg = \mathring{\fg}\otimes_\CC \CC[\epsilon, \epsilon^{-1}] + \CC K +\CC d$
be the corresponding affine Kac-Moody Lie algebra (see \cite[\S6.2]{Kac} -- we follow the notation
of \cite[(3.17)]{LRS}).
Let $\ell\in \ZZ_{>0}$ and let $h$ be the dual Coxeter number.
As explained in \cite[Theorem 3.2]{LRS}, an important result of Kazhdan-Lusztig establishes a
relation between level $(-\ell-h)$-representations in parabolic 
category $\cO_{\mathring{\fg}}^\fg$ for the affine Lie algebra 
and the finite dimensional representations
of the quantum group $U_q\mathring{\fg}$ with $q^{2\ell}=1$.

Let
$$\fg' = [\fg, \fg] = \mathring{\fg}\otimes_\CC \CC[\epsilon,\epsilon^{-1}] + \CC K.$$
By restriction, the modules in $\cO_{\mathring{\fg}}^\fg$ are $\fg'$-modules.  Let $\Lambda_0$ be
the fundamental weight of the affine Lie algebra so that $L(c\Lambda_0+\lambda)$ is an
irreducible highest weight $\fg$-module of level $c$ (i.e. $K$ acts by the constant $c$).

\begin{thm} \label{negleveltoquantum}
\cite[Theorem 38.1]{KL94}
There is an equivalence of categories
\begin{equation*}
\begin{matrix}
\left\{
\begin{matrix}
\hbox{finite length $\fg'$-modules} \\
\hbox{of level $-\ell-h$ in $\cO_{\mathring{\fg}}^\fg$} 
\end{matrix}
\right\}
&\stackrel{\Psi_1}{\longrightarrow}
&\left\{
\begin{matrix}
\hbox{finite dimensional $U_q(\mathring{\fg})$-modules}\\
\hbox{with $q^{2\ell}=1$}
\end{matrix}
\right\}
\\
\Delta_{\mathring{\fg}}^\fg((-\ell-h)\Lambda_0+\lambda)
&\longmapsto
&\Delta_q(\lambda) \\
L((-\ell-h)\Lambda_0+\lambda)
&\longmapsto
&L_q(\lambda)
\end{matrix}
\end{equation*}
\end{thm}

\noindent
This statement of Theorem \ref{negleveltoquantum} is for the simply-laced (symmetric) case.  With the proper 
modifications to this statement the result holds for non-simply laced
cases as well, see \cite[\S 8.4]{Lu94} and \cite{Lu95}.

Let $\lambda\in \fa_\ZZ^*$.
Under the composition of the map $\Psi_1$ in Theorem \ref{negleveltoquantum} and
the map $\Psi_2$ from Theorem  \ref{qgrouptoFockspace},
$$\Psi_2(\Psi_1([L((-\ell-h)\Lambda_0+\ell\lambda)])) = \Psi_2([L_q(\ell\lambda)]) 
= C_{\ell\lambda} = \vert \ell\lambda \rangle.$$
Thus it follows from Theorem \ref{negleveltoquantum}, Theorem \ref{qgrouptoFockspace} and \eqref{KLisket} that
\begin{equation}
L((-\ell-h)\Lambda_0+\ell\lambda) = \Delta_{\mathring{\fg}}^\fg((-\ell-h)\Lambda_0 +\ell\lambda)
= \Ind_{\mathring{\fg_0}+\fb}^{\fg}(L_{\mathring{\fg}}(\ell\lambda))
\cong U\fg \otimes_{U\fk} L_{\mathring{\fg}}(\ell\lambda),
\label{affineFtwistmodule}
\end{equation}
where
$$
\fk = \bigoplus_{k\in \ZZ_{\ge 0}} \epsilon^k
\Big(\fa\oplus 
\bigoplus_{\alpha\in R^+} \mathring{\fg}_\alpha+\mathring{\fg}_{-\alpha}\Big)
\quad\hbox{with $R^+$ the set of positive roots of $\mathring{\fg}$.}
$$
As given in \eqref{Weylchdefn}, the Weyl character formula for the 
$\mathring{\fg}$-module $L_{\mathring{\fg}}(\ell\lambda)$ is
\begin{equation}
\mathrm{char}(L_{\mathring{\fg}}(\ell\lambda)) = s_{\ell\lambda}
= \left(\prod_{\alpha\in R^+}\frac{1}{1-X^{-\alpha}}\right)
\cdot 
\sum_{w\in W_0} \det(w)X^{w\circ \ell\lambda}.
\end{equation}
Letting $q = e^\delta$ and using the Poincar\'e-Birkhoff-Witt theorem,
the character of the $\fg$-module in \eqref{affineFtwistmodule} is
\begin{align}
&\mathrm{char}(L((-\ell-h)\Lambda_0+\ell\lambda)) 
=\mathrm{char}(\Delta_{\mathring{\fg}}^\fg((-\ell-h)\Lambda_0 +\ell\lambda))
= \mathrm{char}(U\fg \otimes_{U\fk} L_{\mathring{\fg}}(\lambda)) 
\nonumber \\
&= 
s_{\ell\lambda} 
\prod_{k\in \ZZ_{>0}} \left( \frac{1}{(1-q^{-k})^n}
\prod_{\alpha\in R^+} \frac{1}{1-q^{-k}X^\alpha}
\cdot \frac{1}{1-q^{-k}X^{-\alpha}}\right)
 \\
&=
\left(\prod_{k\in \ZZ_{>0}} \frac{1}{(1-q^{-k})^n}\right)
\left(\prod_{k\in \ZZ_{>0}} \prod_{\alpha\in R^+} \frac{1}{1-q^{-k}X^\alpha}\right)
\left(\prod_{k\in \ZZ_{\ge 0}} \prod_{\alpha\in R^+} \frac{1}{1-q^{-k}X^{-\alpha}}\right)
\left(\sum_{w\in W_0} \det(w) X^{w\circ \ell\lambda}\right).
\nonumber 
\end{align}
This formula is reminiscent of Weyl-Kac character formula for integrable representations, 
but we have not yet found a reference for it in the literature.  
As we have explained in \eqref{affineFtwistmodule}, this formula is an easy consequence of \cite{KL94} and \cite{Lu89}.

The equivalence in Theorem \ref{negleveltoquantum} is an equivalence of monoidal categories where
the product on the left hand side is the fusion tensor product $\hat\otimes$
and the product on right hand side is the tensor product coming from the Hopf algebra 
structure of $U_q\mathring{\fg}$.
Thus, in terms of affine Lie algebra representations, the Lusztig-Steinberg tensor product theorem
says that
$$\hbox{if $\lambda\in (\fa_\ZZ^*)^+$ \quad and \quad
$\lambda = \lambda_0+\ell\lambda_1$ with $\lambda_0\in \Pi_\ell$}$$
where $\Pi_\ell$ is as in \eqref{Pielldefn}, then
\begin{align}
L((-\ell-h)\Lambda_0+\lambda)
&\cong L((-\ell-h)\Lambda_0+\lambda_0)\hat\otimes L((-\ell-h)\Lambda_0+\ell\lambda_1)
\nonumber \\
&\cong L((-\ell-h)\Lambda_0+\lambda_0)\hat\otimes 
\Delta_{\mathring{\fg}}^{\fg}((-\ell-h)\lambda_0+\ell\lambda_1).
\end{align}

%

\end{document}